\documentclass[12pt]{amsart}

\usepackage[utf8]{inputenc}
\usepackage[french]{babel}
\usepackage{hyperref}
\usepackage{comment}

\usepackage{amssymb,amsthm,amsmath,mathrsfs,amssymb}
\usepackage{euler}

\relax

\renewcommand{\leq}{\leqslant}
\renewcommand{\geq}{\geqslant}

\makeatletter

\def\paragraph{\@startsection{paragraph}{4}%
  \z@\z@{-\fontdimen2\font}%
  {\normalfont\bfseries}}
\makeatother

\newcommand{\Z}{\mathbb{Z}}

\newcommand{\Hp}{\tilde{\mathscr{F}}}
\newcommand{\R}{\mathbb{R}}
\newcommand{\C}{\mathbb{C}}

\newcommand*\colvec[3][]{
    \begin{pmatrix}\ifx\relax#1\relax\else#1\\\fi#2\\#3\end{pmatrix}
}
\newcommand*\abs[1]{
    \left|#1\right|
}

\newcommand{\norm}[1]{\left\lVert#1\right\rVert}
\newtheorem{thm}{Théorème}
\newtheorem{lem}{Lemme}

\newtheorem{coro}{Corollaire}
\newtheorem{prop}{Proposition}

\theoremstyle{definition}
\newtheorem{exemple}{Exemple}
\newtheorem{rem}{Remarque}
\newtheorem{defn}{Définition}
\newtheorem{notation}{Notation}
\newtheorem*{fait}{Fait}
\title[Hypersurfaces Levi-plates linéaires dans les surfaces K3]{Presque toute surface K3 contient une infinité d'hypersurfaces Levi-plates linéaires}
\author{Félix Lequen}
\address{Laboratoire AGM -- CY Cergy Paris Universit\'e}
\email{felix.lequen@cyu.fr}
\usepackage{fullpage}
\begin{document}
\maketitle
\begin{abstract}
On s'intéresse à la construction d'hypersurfaces Levi-plates analytiques réelles dans les surfaces K3. On peut en construire dans les tores complexes de dimension $2$ en prenant des images d'hyperplans réels. On montre que « presque toute » surface K3 contient une infinité d'hypersurfaces Levi-plates de ce type. La preuve repose principalement sur une construction récente due à Koike-Uehara, ainsi que sur les idées de Verbitsky sur les structures complexes ergodiques et une adaptation d'un argument dû à Ghys dans le cadre de l'étude de la topologie des feuilles génériques.
\end{abstract}
\section{Introduction}
On s'intéresse ici à la construction d'hypersurfaces Levi-plates analytiques réelles. Une hypersurface Levi-plate dans une surface complexe est une hypersurface réelle $H$ telle que la distribution dite de Cauchy-Riemann $TH \cap iTH$ soit intégrable. De façon équivalente, localement, l'hypersurface est définie par $\mathrm{Im}\,w = 0$ dans un système de coordonnées locales holomorphes $(z, w)$. La classification des hypersurfaces Levi-plates dans différentes classes de surfaces complexes soulève de nombreuses questions. Par exemple, on ne sait pas si toute surface algébrique contient une hypersurface Levi-plate. On conjecture cependant qu'il n'existe pas d'hypersurface Levi-plate dans le plan projectif complexe $\mathbb{P}^2(\C)$ \cite{cerveau-minimaux} : cette conjecture, issue de la théorie des feuilletages holomorphes, est d'ailleurs l'une des motivations de la notion. Lins Neto \cite{lins-neto} a montré l'inexistence d'hypersurfaces Levi-plates analytiques réelles dans l'espace projectif complexe $\mathbb{P}^n(\C)$ ($n \geq 3$). Citons aussi par exemple l'article de survol \cite{ohsawa-survey}, et les articles \cite{barrett1992topology}, \cite{inaba1994real}, \cite{nemirovskii1999stein}, \cite{deroin2016topology}, \cite{canales} et \cite{adachi2020levi} pour diverses constructions et résultats d'inexistence. Notons en particulier \cite{ohsawa2006levi} qui construit des exemples différents de ceux que nous allons étudier dans les surfaces de Kummer.

Récemment, Koike et Uehara \cite{koike2019gluing} ont construit des exemples remarquables d'hypersurfaces Levi-plates dans des surfaces K3, à partir d'un procédé de type chirurgie, reposant sur un théorème de voisinage tubulaire en contexte holomorphe, remarqué par Arnol'd \cite{arnol1976bifurcations}, en analogie avec des résultats de linéarisation de germes de difféomorphismes holomorphes. Cette construction fournit des familles de surfaces K3 marquées avec 19 modules, qui sont naturellement paramétrées par leurs périodes. Ces hypersurfaces analytiques réelles sont en fait \emph{linéaires}, c'est-à-dire qu'elles sont isomorphes au quotient de $\C \times \R$ par un réseau de rang $3$, où la notion d'isomorphisme est celle d'isomorphisme CR analytique réel, que l'on définira précisément dans notre cas. Avec cette définition, un tore complexe de dimension $2$ contient de nombreuses hypersurfaces Levi-plates linéaires. En effet, dans le cas où l'on considère $\C^2/\Lambda$ où $\Lambda$ est un réseau, tout hyperplan de $\C^2$ dont l'équation est à coefficients rationnels par rapport à $\Lambda$ fournit un exemple.

Le but de ce texte est d'utiliser la construction de Koike-Uehara pour montrer que presque toute surface K3 — en un sens à préciser — contient, comme un tore complexe, de nombreuses hypersurfaces Levi-plates linéaires. Précisons cela. À une hypersurface Levi-plate linéaire d'une surface K3 $X$, on peut associer une classe dite \emph{classe d'Ahlfors}, associée au feuilletage donné par la distribution de Cauchy-Riemann. Ces classes définissent des éléments du groupe de cohomologie $H^{1, 1}(X, \R)$ de $X$, et celles-ci appartiennent au bord du cône associé à la forme d'intersection. Le résultat que l'on prouve est alors le suivant :

\begin{thm}\label{coro-ahlfors}Pour presque tout période $\xi$ dans l'espace des périodes des surfaces K3, et toute surface K3 marquée $(X, \phi)$ de période $\xi$, la surface $X$ contient une infinité d'hypersurfaces Levi-plates linéaires, telles que de plus les classes d'Ahlfors associées sont denses dans la quadrique $\{q = 0\} \cap \mathbb{P}(H^{1, 1}(X, \R))$ de $\mathbb{P}(H^{1, 1}(X, \R))$, où $q$ désigne la forme d'intersection de $X$.
\end{thm}

On précisera les définitions, et notamment le sens exact de « presque toute période » à la section \ref{section:infinite-hlpl}. Le point de départ de la preuve est que la construction de Koike-Uehara fournit un ensemble de mesure positive dans l'espace des périodes des surfaces K3,  comme on le verra, ce qui permet d'appliquer des idées dues à Verbitsky \cite{verbitsky2015ergodic} sur les structures complexes ergodiques. L'espaces des périodes s'identifie en effet à un espace homogène du groupe de Lie $SO(3, 19)$, sur lequel agit un groupe modulaire qui s'identifie à l'action par translation d'un réseau de $SO(3, 19)$. Contrairement au cas des courbes où le groupe modulaire agit de façon proprement discontinue sur l'espace de Teichmüller, qui est ici l'analogue de l'espace des périodes, l'action est ergodique, par le théorème de Howe-Moore \cite{zimmer2013ergodic} un résultat de dynamique homogène que nous aurons besoin de raffiner.

Enfin, pour étudier la densité des classes d'Ahlfors, on remarque que le cône donné par la forme d'intersection de $X$ s'identifie au bord de l'espace hyperbolique de dimension 19, et on utilise des arguments de géométrie hyperbolique, adaptés d'une idée de Ghys \cite{ghys1995topologie}.

Comme conséquence, on obtient un résultat s'énonçant sans théorie de la mesure, que l'on déduit du théorème \ref{coro-ahlfors} et de la surjectivité de l'application des périodes \ref{surjectivite-periode}, comme dans la section 5 de \cite{koike2019gluing} :
\begin{coro}
	Il existe une surface K3 $X$ qui n'est ni projective ni une surface de Kummer et qui contient une infinité d'hypersurfaces Levi-plates linéaires, telles que de plus les classes d'Ahlfors associées sont denses dans la quadrique $\{q = 0\} \cap \mathbb{P}(H^{1, 1}(X, \R))$ de $\mathbb{P}(H^{1, 1}(X, \R))$, où $q$ désigne la forme d'intersection de $X$.
\end{coro}
À notre connaissance, on ne sait pas s'il existe une surface K3 qui ne contient pas d'hypersurface Levi-plate linéaire analytique réelle.

\paragraph{Organisation de l'article. } La section \ref{section:debut} donne des notations et des définitions sur les hypersurfaces Levi-plates linéaires, puis reproduit brièvement la construction de Koike-Uehara. La section \ref{section:infinite-hlpl} montre que presque toute surface K3 contient une infinité d'hypersurfaces Levi-plate linéaire, ce qui est nécessaire pour la preuve de la proposition \ref{coro-ahlfors}, qui est donnée dans le section \ref{section:ahlfors}.
\paragraph{Remerciements. } Je remercie vivement mon directeur de thèse, Bertrand Deroin, pour m'avoir proposé cette approche et guidé tout au long de sa mise en œuvre. Entre autres choses, l'argument permettant de montrer la densité des classes d'Ahlfors lui est dû.
\section{La construction de Koike-Uehara}
\label{section:debut}
Dans cette section, on commence par quelques rappels, notations, et définitions nécessaires à la suite, puis on reproduit brièvement la construction de Koike et Uehara.
\paragraph{Fibrés plats sur une courbe elliptique. }
Soit $C$ une courbe elliptique lisse. Notons $U(1) := \{z\in \C\,:\,\abs{z} = 1\}$ le cercle unité vu comme groupe. Étant donné une représentation $\rho \colon \pi_1(C) \to U(1)$ et $V \subset \C$ un sous-ensemble du plan complexe invariant par l'action de $U(1)$ par multiplication, soit $C \times_\rho V$ le quotient de $\tilde{C} \times V$ par l'action de $\pi_1(C)$ agissant diagonalement via $\rho$ : pour tout $\gamma \in \pi_1(C)$, l'action est donnée par
\[\gamma \cdot (\zeta, v) := (\gamma \cdot \zeta, \rho(\gamma)v).\]
Si $V \subset W$, on identifiera $C \times_\rho V$ à un sous-ensemble de $C \times_\rho W$. Il sera utile par la suite de poser les notations suivantes, pour $R > 0, r^+ > r^- > 0$ :
\[\mathbb{S}^1(R) = \{v\in\C\;:\;|v| = R\} \text{ et } \mathbb{S}^1 := \mathbb{S}^1(1),\]
\[\mathbb{D}(R) = \{v\in\C\;:\;|v| < R\},\]
\[\mathbb{A}(r_{-}, r_{+}) = \{v\in\C\;:\;r_{-}<|v|<r_{+}\}.\]

La projection $(\zeta, v) \mapsto \zeta$ induit une application $C \times_\rho V \to C$ qui munit $C \times_\rho V$ d'une structure de fibré, de fibre $V$. De plus, le fibré $C \times_\rho \C$ au-dessus de $C$ est naturellement un fibré en droites holomorphe. Il est connu que tous les fibrés en droites holomorphes topologiquement triviaux sur une courbe complexe lisse sont de la forme précédente et de plus $C \times_\rho \C$ et $C \times_{\rho'} \C$ sont isomorphes en tant que fibrés en droites holomorphes si et seulement si $\rho = \rho'$.
On notera dans la suite $L_\rho$ le fibré en droites holomorphe $C \times_\rho \C$.

\paragraph{Théorème de linéarisation d'Arnol'd. } On donne ici l'énoncé du théorème de linéarisation d'Arnol'd \cite{arnol1976bifurcations}, nécessaire à la construction de Koike-Uehara. C'est un équivalent holomorphe du théorème du voisinage tubulaire classique dans le cadre différentiel, nécessaire pour un procédé de chirurgie. Il ne s'applique cependant pas dans tous les cas : on a besoin d'une hypothèse diophantienne, qui est vraie seulement génériquement.
\begin{defn}
Soit $C, \alpha > 0$. On note
\[\mathrm{Dioph}(C, \alpha) = \left\{(p, q) \in \R^2\,:\,\norm{np}_{\R/\Z} + \norm{nq}_{\R/\Z} > Cn^{-\alpha} \text{ pour tout entier } n \geq 1\right\}\]
et
\[\mathrm{Dioph} = \bigcup_{C, \alpha > 0} \mathrm{Dioph}(C, \alpha).\]
\end{defn}
Ici, on a noté $\norm{x}_{\R/\Z}$ la distance à $\Z$ d'un réel $x$. Il n'est pas difficile de voir que $\mathrm{Dioph}$ est de mesure pleine.
\begin{defn}
On dit qu'une représentation $\rho\colon L\to U(1)$ d'un $\Z$-module $L \simeq \Z^2$ est \emph{diophantienne} s'il existe $(p, q) \in \mathrm{Dioph}$ et $(\tau_1, \tau_2)$ une base de $L$ telle que $\rho(\tau_1) = e^{2i\pi p}$, $\rho(\tau_2) = e^{2i\pi q}$.
\end{defn}

\begin{thm}[Arnol'd \cite{arnol1976bifurcations}]
Soit $S$ une surface complexe et $C \subset S$ une courbe elliptique compacte lisse. Supposons que le fibré normal de $C$ dans $S$ soit isomorphe à $L_\rho$ pour $\rho$ une représentation diophantienne. Alors la courbe elliptique $C$ est \emph{linéarisable} : il existe un biholomorphisme d'un voisinage ouvert de $C$ dans $S$ dans un voisinage ouvert de la section nulle dans $L_\rho$ qui envoie $C$ sur la section nulle.
\end{thm}

\paragraph{Hypersurfaces Levi-plates. }
Soit $X$ une variété complexe.
\begin{defn}
Une hypersurface $H$ de $X$ est une \emph{hypersurface Levi-plate} (analytique réelle) si elle admet un feuilletage (analytique réel) de codimension réelle un dont les feuilles sont des sous-variétés complexes immergées de $X$. Ce feuilletage est appelé le \emph{feuilletage de Cauchy-Riemann} de $H$.
\end{defn}
\begin{rem}
Sauf mention du contraire, toutes les hypersurfaces Levi-plates sont supposées analytiques réelles dans la suite.
\end{rem}
\begin{exemple}
\begin{enumerate}
\item L'hypersurface $\C^n\times\R \subset \C^{n + 1}$ est naturellement une hypersurface Levi-plate, pour tout entier $n$.
\item Soit $V \subset \C^2$ un hyperplan réel, et $L$ un réseau de $\C^2$ de rang $4$ tel que $V$ soit rationnel par rapport à $L$, c'est-à-dire que $V$ admet une base de vecteurs contenus dans $L$. Alors l'image de $V$ dans $\C^2/L$ est naturellement une hypersurface Levi-plate, difféomorphe à un tore de dimension $3$ réelle.
\end{enumerate}
\end{exemple}
\begin{defn} \begin{enumerate}
\item Une hypersurface Levi-plate (analytique réelle) $H \subset X$ est dite \emph{Levi-plate linéaire} (HLPL) si elle est CR-isomorphe (de façon analytique réelle) au quotient de $\C\times\R$ par un réseau de rang $3$ (voir le livre \cite{CRStructures} pour plus de détails). Cela signifie qu'il existe un difféomorphisme (analytique réel) $\phi$ de $H$ vers $\R^3/\Z^3$ tel que
\begin{enumerate}
\item le feuilletage de Cauchy-Riemann de $H$ est envoyé sur un feuilletage linéaire de $\R^2/\Z^3$, c'est-à-dire dont les feuilles sont les images de $u + V$ lorsque $u \in\R^3$ varie, pour un hyperplan fixé $V$ de $\R^3$, 
\item il existe une structure complexe $J$ sur $V$ (c'est-à-dire un endomorphisme $J$ de $V$ tel que $J^2 = -\mathrm{id}_V$) tel que pour toute feuille $F$ du feuilletage de Cauchy-Riemann de $H$, on a $d\phi_{|F} \circ J' = J \circ d\phi_{|F}$, où $J'$ désigne la structure complexe de $X$ (on a identifié les espaces tangents à $\phi(F)$ à $V$).
\end{enumerate}

De façon équivalente, $H$ peut être obtenu comme dans l'exemple précédent : $H$ est une HLPL s'il existe
\begin{enumerate}
\item un hyperplan réel $V$ de $\C^2$ ;
\item un réseau $L$ de rang 4 de $\C^2$ tel que $V$ soit rationnel par rapport à $L$ ;
\item un isomorphisme analytique réel $\phi\colon H \to H'$ où $H'$ est l'image de $V$ dans $\C^2/L$
\end{enumerate}tels que $\phi$ s'étend en une application biholomorphe d'un voisinage de $H$ dans un voisinage de $H'$. L'équivalence des définitions découle de l'analyticité réelle des objets considérés, qui permet en quelque sorte de reconstruire la structure complexe au voisinage de $H$.
\item Une hypersurface Levi-plate linéaire est dite \emph{irrationnelle} (HLPLI) si les feuilles de son feuilletage de Cauchy-Riemann sont denses.
\end{enumerate}
\end{defn}

L'exemple suivant est important pour la suite.
\begin{exemple}
\label{hlpl-elliptique}
Soit $C = \C/(\Z\omega_1 + \Z\omega_2)$ une courbe elliptique, où $\omega_1$ et $\omega_2$ engendrent un réseau de $\C$. Soit $\rho \colon \pi_1(C) \to U(1)$ une représentation. Considérons $L = C \times_\rho \mathbb{S}^1$ où $\mathbb{S}^1 := \mathbb{S}^1(1)$. Montrons que c'est une hypersurface Levi-plate linéaire de $C \times_\rho \mathbb{A}(e^{-\pi}, e^{\pi})$.

Soit $p, q \in \R$ tels que $e^{2i\pi p} = \rho(\omega_1)$ et $e^{2i\pi q} = \rho(\omega_2)$. Considérons $L'$ le réseau de $\C^2$ engendré par $(\omega_1, p)$, $(\omega_2, q)$, $(0, 1)$ et un autre vecteur pour que ce soit bien un réseau de rang $2$, disons $(0, i)$ pour la suite. 
Alors l'application $(\zeta, v) \mapsto (\zeta, e^{2i\pi v})$ passe au quotient en un biholomorphisme de $\{[(\zeta, v)]\,:\,\abs{\mathrm{Im}\,v} < \frac{1}{2}\}\subset\C^2/L'$ dans $C \times_\rho \mathbb{A}(e^{-\pi}, e^{\pi})$ qui envoie $\{\mathrm{Im}\,v = 0\}$ sur $C\times_\rho \mathbb{S}^1$. C'est donc bien une hypersurface Levi-plate linéaire. Plus généralement, à partir de toute surface complexe contenant une courbe elliptique linéarisable, on peut construire ainsi une hypersurface Levi-plate linéaire analytique réelle de cette façon.
\end{exemple}

\paragraph{La construction de Koike-Uehara. } On rappelle ici la construction de Koike-Uehara \cite[section 2]{koike2019gluing}.
On considère $C_0^{\pm} \subset \mathbb{P}^2(\C)$ deux cubiques planes lisses et $p_1^{\pm}, p_2^{\pm}, \ldots, p_9^{\pm} \in C_0$. Soient $S^{\pm}$ les éclatés de $\mathbb{P}^2(\C)$ en ces $9$ points ; on note $E_1^{\pm}, E_2^{\pm}, \ldots, E_9^{\pm}$ les diviseurs exceptionnels associés et $C^{\pm}$ les relevées strictes de $C_0^{\pm}$ dans $S^{\pm}$. 

Les diviseurs canoniques $K_{S^{\pm}}$ de $S^{\pm}$ vérifient $K_{S^{\pm}} = \left(\sigma^{\pm}\right)^{*}(-3H) + E_1 + \ldots + E_9$ où $H$ désigne une section hyperplane dans $\mathbb{P}^2(\C)$ et $\sigma^{\pm}$ sont les applications d'éclatement. Notons d'autre part que $C^{\pm} = \left(\sigma_{\pm}\right)^{*}C_0^{\pm} - (E_1^{\pm} + \ldots + E_9^{\pm}) = \left(\sigma^{\pm}\right)^{*}(3H)-(E_1^{\pm} + \ldots + E_9^{\pm})$, donc $K_{S^{\pm}} = -C^{\pm}$. Enfin, par la formule d'adjonction, on a $0 = K_{C^{\pm}} = (K_{S^\pm})_{|C^{\pm}} + N^{\pm}$ où $N^{\pm}$ désigne le diviseur du fibré normal de $C^{\pm}$ dans $S^{\pm}$, donc via l'isomorphisme $C^{\pm} \to C_0^{\pm}$, le diviseur de $N^{\pm}$ est $3H_{|C_0^{\pm}}- (p_1^{\pm} + \ldots + p_9^{\pm})$. 

Notons que le degré de $N^{\pm}$ est nul ; ce sont donc des fibrés plats. Par la section précédente, ils sont définis par des représentations $\pi_1(C^\pm) \to U(1)$.
Si les fibrés normaux $N^{\pm}$ aux courbes $C^\pm$ dans $S^\pm$ sont donnés par des représentations diophantiennes $\rho^\pm$, on peut appliquer le théorème de linéarisation d'Arnol'd. On en déduit deux voisinages $V^\pm \subset S^\pm$ isomorphes à $C^\pm\times_{\rho^\pm} \mathbb{D}(r^\pm)$ par des biholomorphismes qui envoient $C^\pm$ sur $C^{\pm} \times \{0\}$. On identifie par la suite $V^\pm$ à $C^\pm\times_{\rho^\pm} \mathbb{D}(r^\pm)$.

Supposons maintenant que l'on dispose d'un biholomorphisme $g \colon C^{+} \to C^{-}$ tel que $g^{*}N^{-} \cong \left(N^{+}\right)^{-1}$. Alors on a $g^*\rho^{-} = \frac{1}{\rho^{+}}$, où par définition $(g^{*}\rho^{-})(\gamma) = \rho^{-}(g_{*}(\gamma))$ pour tout $\gamma \in \pi_1(C^+)$ et $g_* \colon \pi_1(C^+) \to \pi_1(C^-)$ est l'application induite sur le groupe fondamental.

Considérons alors l'application $\tilde{g} \colon \tilde{C^+} \to \tilde{C^-}$ induite en passant au revêtement universel. Posons $V^+ := C^+\times_{\rho^+} \mathbb{A}(|t|/r^-, r^+)$ et $V^- := C^-\times_{\rho^-} \mathbb{A}(|t|/r^+, r^-)$ et définissons $F \colon V^+ \to V^-$ induite par\[
(\zeta, v) \mapsto (\tilde{g}(\zeta), t/v),\]où $t\in\C$ vérifie $0 < |t| < r^{-}r^{+}$.

L'application $F$ permet de recoller $M^+ := S^+\setminus(C^+\times_{\rho^{+}}\mathbb{D}(0, |t|/r^-))$ à $M^- := S^-\setminus(C^-\times_{\rho^{-}}\mathbb{D}(0, |t|/r^+))$ le long de $V^+$ et $V^-$, ce qui permet de définir une nouvelle variété complexe notée $X$, dans laquelle $M^+$ et $M^-$ sont des ouverts. Les ouverts $V^+$ et $V^-$ sont alors identifiés à un ouvert $V \subset X$. Enfin, notons que $V^\pm$ est remplie par des hypersurfaces Levi-plates, à savoir les hypersurfaces $C^+\times_{\rho^+} \mathbb{S}^1(R)$ pour $|t|/r^- < R < r^+$. Elles sont irrationnelles car les représentations sont diophantiennes.

\begin{prop}La surface complexe $X$ est une surface K3, et il existe une $2$-forme holomorphe sur $X$ notée $\sigma$ telle que, dans les coordonnées données par l'identification $V \simeq V^+ = C^+\times_{\rho^+} \mathbb{A}(|t|/r^-, r^+)$, on ait
\[\sigma_{|V^+} = d\zeta \wedge \frac{dv}{v}.\]
\end{prop}
\begin{proof}On a vu plus haut que $K_{S^\pm} = -C^\pm$, de sorte qu'il existe une $2$-forme méromorphe $\eta^\pm$ sur $S^\pm$, qui ne s'annule pas et avec un pôle simple le long de $C^\pm$. Alors l'expression
\[\frac{\eta^\pm}{d\zeta \wedge \frac{dv}{v}}\]
définit une application holomorphe sur $V^\pm$. Or la densité des feuilles des hypersurfaces Levi-plates et le principe du maximum permettent de montrer que cette application est constante (voir \cite[lemme 2.2]{koike2019gluing}) ; notons $A^\pm$ sa  valeur. Comme on a $F^*(d\zeta \wedge \frac{dv}{v}) = -d\zeta\wedge\frac{dv}{v}$, les $2$-formes $\frac{\eta^+}{A^+}$ et $-\frac{\eta^-}{A^-}$ se recollent en une $2$-forme $\sigma$ qui ne s'annule pas \cite[proposition 2.1]{koike2019gluing}. Dans la suite, on normalisera $\eta^\pm$ pour que $A^\pm = 1$.

Enfin, la suite exacte de Mayer-Vietoris permet de montrer que $X$ est simplement connexe \cite[proposition 2.1]{koike2019gluing}, donc $X$ est bien une surface K3.

\end{proof}

À ce stade, il est utile de noter que Koike-Uehara construisent géométriquement une base du deuxième groupe d'homologie à partir des données du recollement \cite[section 3]{koike2019gluing}. Cela sera utile par la suite, pour obtenir des surfaces K3 marquées. Nous y reviendrons. 
\paragraph{Construction de familles. }
L'intérêt de la construction de Koike-Uehara est de construire des familles de surfaces K3. Fixons $(p,q) \in \mathrm{Dioph}$. Reprenons les notations du paragraphe précédent.
\begin{enumerate}
\item Soit $t \in \C$ avec $|t| < r_{-}r_{+}$. Alors on peut recoller les extérieurs via l'application
\[\C \times_{\rho_{+}} \mathbb{A}(|t|/r_{-}, r_{+}) \to \C \times_{\rho_{-}} \mathbb{A}(|t|/r_{+}, r_{-})\]induite par
\[(\zeta, v) \mapsto (\tilde{g}(\zeta), t/v).\]
Cela donne un paramètre complexe.
\item On peut changer l'isomorphisme $g$ ci-dessus, en lui ajoutant une translation, ce qui donne un autre paramètre complexe.
\item On peut changer la classe d'isomorphisme de la courbe elliptique de départ, ce qui donne un paramètre complexe $\tau \in \mathbb{H} :=\{z\in \C\;:\;\mathrm{Im}\,z > 0\}$.
\item On peut faire varier 8 des 9 points choisis sur la courbe elliptique. Le neuvième point est alors choisi de façon à ce que si $C_0$ est isomorphe à $\C/(\Z + \Z\tau)$, alors via l'isomorphisme entre la jacobienne de $C_0$ et $\C/(\Z + \Z\tau)$, le fibré normal de $C^{+}$ soit égal à $p - q\tau$.
On peut faire la même chose pour $C^{-}$, avec un fibré normal cette fois égal à $-p + q\tau$.  Cela donne deux fois 8 paramètres complexes.
\end{enumerate}
Au total on a donc 19 paramètres complexes. On renvoie à l'article de Koike-Uehara \cite[section 4]{koike2019gluing} pour une construction détaillée des déformations.

\section{Existence d'une infinité d'hypersurfaces Levi-plates linéaires}
\label{section:infinite-hlpl}
Dans cette section, on montre qu'il existe un ensemble de mesure pleine de périodes dont les surfaces K3 marquées associées sont, à un changement de marquage près, des surfaces de Koike-Uehara d'une infinité de façons différentes, et en particuliers contiennent une infinité d'hypersurfaces Levi-plates linéaires analytiques réelles. Ce résultat est plus faible que le théorème \ref{coro-ahlfors}, mais il est nécessaire à sa preuve. 
\paragraph{Surfaces K3 et périodes. }
On rappelle ici pour fixer les notations les définitions liées aux périodes de surfaces K3 dont on fera usage (voir le livre de Huybrechts pour plus d'informations \cite[chapitres 6 et 7]{huybrechts_2016}). On note $\Lambda$ le réseau K3, c'est-à-dire $U^{\oplus 3} \oplus E_{8}(-1)^{\oplus 2}$, 
muni de sa forme quadratique notée $q(\cdot, \cdot)$.

On aura besoin d'éléments explicites de $\Lambda$ adaptés à notre situation. Considérons $A_{\alpha\beta}$, $A_{\beta\gamma}$, $A_{\gamma\alpha}$, $B_\alpha$, $B_\beta$ et $B_\gamma$ des éléments de $\Lambda$ qui engendrent un réseau $\Lambda'$ isométrique à $U^{\oplus 3}$, où $U$ est un plan hyperbolique. Plus précisément, on suppose que
\begin{align*}q(A_{\alpha\beta}, A_{\alpha\beta}) &= 0, \\
q(A_{\alpha\beta}, B_{\gamma}) &= 1, \\
q(B_{\gamma}, B_\gamma) &= -2.\end{align*}
et de façon analogue pour les permutations cycliques de $\alpha, \beta, \gamma$. On suppose alors que $\Lambda = \Lambda' \oplus \Lambda''$, la somme étant orthogonale, où $\Lambda''$ est isométrique à $E_{8}(-1)^{\oplus 2}$. Enfin, on note $a_{\alpha\beta}, a_{\beta\gamma},a_{\gamma\alpha}, b_\alpha,b_\beta, b_{\gamma}$ les formes linéaires sur $\Lambda\otimes\C$ obtenues en étendant par $0$ la base duale de $A_{\alpha\beta}$, $A_{\beta\gamma}$, $A_{\gamma\alpha}$, $B_\alpha$, $B_\beta$ et $B_\gamma$ de $\Lambda'\otimes\C$.

\begin{defn}On appelle \emph{surface K3 marquée} un couple $(X, \phi)$ où $X$ est une surface K3, et $\phi \colon H^2(X, \Z) \to \Lambda$ est une isométrie, où $H^2(X, \Z)$ est muni de sa forme d'intersection et $\Lambda$ de la forme $q$. Deux surfaces K3 marquées $(X, \phi)$ et $(X', \phi')$ sont \emph{équivalentes} s'il existe un biholomorphisme $u \colon X \to X'$ tel que $\phi \circ u^* = \phi'$.

Soit
\[\mathscr{D} := \{\C\xi \in \mathbb{P}(\Lambda \otimes\C)\,:\,q(\xi, \xi) = 0,\;q(\xi, \bar{\xi}) > 0\}\]
le \emph{domaine des périodes}. Étant donné une surface K3 marquée $(X, \phi)$, on considère l'application obtenue par tensorisation $\phi_\C \colon H^2(X, \C) \to \Lambda\otimes\C$, et on note $\mathscr{P}(X, \phi) := \phi_\C(H^{2, 0}(X)) \in \mathscr{D}$ la \emph{période} de $(X, \phi)$. On appelle $\mathscr{P}$ \emph{l'application des périodes}.
\end{defn}

On a dit que la construction de Koike-Uehara permettait de construire une base du deuxième groupe d'homologie. En fait, cela permet de construire un marquage. Les surfaces de Koike-Uehara construites dans la section précédente sont donc naturellement des surfaces K3 marquées. On peut même en déduire un calcul des périodes (sauf deux).

\begin{defn}
On appellera \emph{surface de Koike-Uehara} les surfaces K3 marquées construites ainsi. Chacune de ces surfaces contient naturellement des hypersurfaces Levi-plates linéaires irrationnelle, toutes isomorphes, qui sont naturellement associées.
\end{defn}

Si $(X, \phi)$ est une telle surface, les éléments $\phi^{-1}(A_{\alpha\beta})$, $\phi^{-1}(A_{\beta\gamma})$ et $\phi^{-1}(A_{\gamma\alpha})$ engendrent un sous-module de $H^2(X, \Z)$ qui est isomorphe à l'image par la dualité de Poincaré de $H_2(L, \Z)$, où $L$ est l'une des hypersurfaces Levi-plates linéaires associées.

Enfin, on aura besoin par la suite des deux théorèmes fondamentaux suivants sur les surfaces K3 marquées et leurs périodes.
\begin{notation}
\label{notation-generique}
Soit $\mathscr{G} := \mathscr{D}\setminus\bigcup_{\alpha \in \Lambda, \alpha \neq 0} \alpha^\perp$ (« ensemble générique »).
\end{notation}

\begin{thm}[Torelli]
\label{torelli}
Soient $(X, \phi)$ et $(X', \phi')$ deux surfaces K3 marquées telles que $\mathscr{P}(X, \phi) = \mathscr{P}(X, \phi')$. Alors $X$ et $X'$ sont biholomorphes. De plus, si $\mathscr{P}(X, \phi) = \mathscr{P}(X, \phi') \in \mathscr{G}$, alors $(X, \phi)$ et $(X', \pm\phi')$ sont équivalentes.
\end{thm}
\begin{proof}Voir le livre de Huybrechts \cite[chapitre 7, proposition 2.1 et 2.2, remarque 5.2]{huybrechts_2016}
\end{proof}

Citons aussi le résultat important suivant (voir encore le livre de Huybrechts \cite[chapitre 7]{huybrechts_2016} par exemple), bien qu'il ne soit pas nécessaire dans ce qui suit :
\begin{thm}[Surjectivité de l'application des périodes]
Soit $\xi \in \mathscr{D}$. Alors il existe une surface K3 marquée $(X, \phi)$ telle que $\mathscr{P}(X, \phi) = \xi$.
\label{surjectivite-periode}
\end{thm}
Le fait bien connu suivant permettra d'utiliser des outils de théorie ergodique pour les espaces de périodes. Notons que l'on utilise pas les résultats explicites de l'approche de Verbitsky \cite{verbitsky2015ergodic} mais uniquement ses idées, dans un cadre technique un peu plus simple.
\begin{prop}
L'espace des périodes $\mathscr{D}$ est difféomorphe à la grassmanienne $\mathrm{Gr}_2^{\mathrm{po}}(\Lambda\otimes\R)$ des plans réels de dimension $2$ définis positifs orientés de l'espace $\Lambda\otimes\R$, qui est difféomorphe à $SO^\circ(3, 19)/(SO(2) \times SO(1, 19))$.
\label{periodes-homogene}
\end{prop}
\begin{proof}
Soit $\xi \in  \mathbb{P}(\Lambda \otimes \C)$. Considérons $x$ et $y$ les parties réelles et imaginaires d'un vecteur de $\xi$ non nul. Alors la relation $q(\xi, \xi) = 0$ entraîne que $q(x, x) = q(y, y)$ et $q(x, y) = 0$. La relation $q(\xi, \overline{\xi})>0$ donne que $q(x, x) + q(y, y) >0$. Donc $(x, y)$ engendre un plan réel défini positif orienté.
Réciproquement, si $P$ est un tel plan, et $(x, y)$ une base orthonormée directe de $P$, alors la droite $\C(x + iy)$ définit un élément de $\mathscr{D}$.

Enfin, l'espace $\mathrm{Gr}_2^{\mathrm{po}}(\Lambda\otimes\R)$ est un espace homogène pour l'action de $SO^\circ(q)$, la composante neutre de $SO(q)$. On peut choisir un isomorphisme $SO^\circ(q) \simeq SO^\circ(3, 19)$ de sorte que le stabilisateur d'un point fixé soit $SO(2)\times SO(1,19)$, plongé naturellement. 
\end{proof}

\begin{notation}
Étant donné $\xi \in \mathscr{D}$, on note $P(\xi) \subset \Lambda \otimes \R$ le plan réel correspondant à $\xi$ via l'isomorphisme précédent.
\end{notation}

\paragraph{Existence d'une infinité d'HLPLI }
Nous aurons besoin du résultat suivant, conséquence des travaux de Koike-Uehara, dont on donne la preuve en annexe.
\begin{coro}
\label{thm-construction}
\label{coro-borelien}
Il existe un borélien $B$ de $\mathscr{D}$ qui n'est pas de mesure nulle dont tous les éléments sont des périodes réalisées par une surface de Koike-Uehara.
\end{coro}
\begin{rem}
Sur une variété lisse, la classe de la mesure de Lebesgue est  bien définie et ainsi la notion de mesure nulle a un sens, de même que la notion de mesure pleine. Sauf mention du contraire, on utilisera cette notion sur les variétés que l'on considère, et en particulier $\mathscr{D}$.
\end{rem}
Soit $\Gamma = O(\Lambda, q) \cap SO^\circ(q)$ le groupe d'isométries du réseau $\Lambda$. C'est un groupe arithmétique, donc un réseau de $SO^\circ(q)$ par le théorème de Borel-Harish-Chandra \cite{borelharishchandra}.

\begin{notation}
On fixe dans toute la suite $B$ un borélien comme dans le corollaire \ref{coro-borelien} et $A := \cup_{\gamma \in \Gamma} \gamma B$.
\end{notation}
\begin{prop}Le borélien $A$ est de mesure pleine.
\end{prop}
\begin{proof}
Par le théorème de Howe-Moore \ref{howe-moore} ci-dessous, l'action de $\Gamma$ sur $\mathscr{D}$, qui s'identifie à l'action d'un réseau de $SO^\circ(3, 19)$ sur $SO^\circ(3, 19)/\left(SO(2)\times SO(1, 19)\right)$ par translation (voir la preuve de la propositon \ref{periodes-homogene}), est ergodique.
\end{proof}

\begin{thm}[Howe-Moore \cite{zimmer2013ergodic}]
Soit $G$ un groupe de Lie connexe presque simple à centre fini. Soit $\Gamma$ un réseau de $G$.

Soit $H$ un sous-groupe d'adhérence non compacte de $G$.

Si $A \subset G/H$ est un borélien invariant par $\Gamma$ à gauche, alors soit $A$ est de mesure nulle, soit $A$ est de mesure pleine.
\label{howe-moore}
\end{thm}

Rappelons qu'un groupe de Lie est \emph{presque simple} si son algèbre de Lie est simple.

\begin{coro}
Pour presque toute période $\xi \in \mathscr{D}$ et toute surface K3 marquée $(X, \phi)$ de période $\xi$, il existe une surface de Koike-Uehara $Y$ qui est biholomorphe à $X$.

En particulier, pour presque toute période $\xi \in \mathscr{D}$ et toute surface $K3$ marquée $(X, \phi)$ de période $\xi$, $X$ contient une hypersurface Levi-plate linéaire analytique réelle.
\end{coro}
\begin{proof}
Soient $\xi \in A$ (l'ensemble $A$ étant de mesure pleine) et $(X, \phi)$ comme dans l'énoncé. Par hypothèse, il existe $\gamma \in \Gamma$ tel que $\gamma \xi \in B$. Alors il existe une surface de Koike-Uehara $(Y, \psi')$ de période $\gamma\xi$. La surface $(Y, \gamma^{-1}\circ\psi')$ a pour période $\xi$, donc par le théorème de Torelli \ref{torelli}, $Y$ est biholomorphe à $X$.
\end{proof}

En fait, on peut faire mieux. Considérons une surface complexe $X$ et $L \subset X$ un tore de dimension $3$. Alors on a une application de $H_2(L, \Z)$ dans $H_2(X, \Z)$ induite par l'inclusion ; par dualité de Poincaré, on peut identifier son image à un sous-module de $H^2(X, \Z)$ de rang au plus $3$ que l'on note $h(L)$. Notons que si $X'$ est une autre surface complexe, et $f \colon X' \to X$ est un biholomorphisme, $f^{-1}(L)$ est encore un tore de dimension $3$ et la naturalité de la dualité de Poincaré montre que $h(f^{-1}(L)) = f^*h(L)$.

 On va définir une application qui à $\xi \in A$ associe l'ensemble des $h(L)$, vus dans le marquage, où $L$ est une HLPL de $X$ obtenue par application d'un élément du groupe $\Gamma$ à une surface de Koike-Uehara. On aura seulement besoin de la proposition-définition suivante. Notons $\mathscr{E}$ l'ensemble — dénombrable — des sous-modules de rang au plus $3$ de $\Lambda$.

\begin{prop}Il existe une application mesurable $\mathscr{H} \colon \mathscr{D} \to \mathscr{P}(\mathscr{E})$ qui est $\Gamma$-équivariante et telle que pour tout $\xi \in \mathscr{G}$, toute surface K3 marquée $(X, \phi)$ de période $\xi$, et tout $M \in \phi^{-1}\left(\mathscr{H}(\xi)\right)$, il existe une hypersurface Levi-plate linéaire irrationnelle $L \subset X$ telle que $h(L) = M$.
\label{prop:equiv-homologie}
\end{prop}
Ici, on a noté $\mathscr{P}(\mathscr{E})$ l'ensemble des parties de $\mathscr{E}$, muni de la tribu engendrée par les parties finies dont les éléments appartiennent à $\mathscr{P}(\mathscr{E})$.

\begin{proof}
Étant donné $\xi \in B$, soit $(X, \phi)$ la surface de Koike-Uehara associée et considérons l'une des HLPLI associées, notée $L$.  On peut considérer le sous-module de $H^2(X, \Z)$ défini plus haut $h(L)$, qui s'identifie via $\phi$ à un sous-module de $\Lambda$. On définit naturellement de la sorte une application $H \colon B \to \mathscr{E}$, où l'on note $\mathscr{E}$ l'ensemble des sous-modules de rang au plus $3$ de $\Lambda$, qui est un ensemble dénombrable. En fait, l'application $H$ est constante, égale à $\Z A_{\alpha\beta} \oplus \Z A_{\beta\gamma} \oplus \Z A_{\gamma\alpha}$.

Étant donné $\xi\in A$, notons $\Gamma_\xi := \{\gamma\in\Gamma\,:\,\gamma\xi \in B\}$ et $\mathscr{H}(\xi) := \{\gamma^{-1}H(\gamma\xi)\,:\,\gamma\in\Gamma_\xi\}$. 

Montrons que l'application $\mathscr{H}$ est mesurable. Soit $e \in \mathscr{E}$. Alors l'ensemble $\mathscr{D}_e := \{\xi \in \mathscr{D}\,:\,e \in \mathscr{H}(\xi)\}$ s'écrit
\[\mathscr{D}_e = \{\xi\in \mathscr{D}\,:\,\text{ il existe } \gamma \in \Gamma \text{ tel que } \gamma\xi \in A \text{ et } \gamma^{-1}H(\gamma\xi) = e\}.\]
C'est donc un ensemble mesurable. Ainsi $\mathscr{H}$ est bien une application mesurable.

Soit $\xi \in \mathscr{G}$ et $(X, \phi)$ une surface K3 marquée de période $\xi$. Soit $M \in \phi^{-1}\left(\mathscr{H}(\xi)\right)$. Alors il existe $\gamma \in \Gamma_\xi$ tel que $\gamma \phi(M) = H(\gamma\xi)$, donc il existe une surface de Koike-Uehara marquée $(Y, \psi)$ de période $\gamma\xi$ qui admet une HLPLI, disons $L \subset Y$, telle que $\psi(h(L)) = \gamma \phi(M)$. Alors $\gamma^{-1}(Y, \psi) = (Y, \gamma\circ\psi)$ a $\xi \in \mathscr{G}$ pour période, donc il existe un biholomorphisme $u \colon X \to Y$ tel que $\gamma^{-1} \circ\psi = \pm\phi \circ u^{*}$. 
Alors on a
\[\phi \circ h(u^{-1}(L)) = \phi(u^{*}(h(L)) = \pm\gamma^{-1}\circ\psi(h(L)) = \pm \gamma^{-1}\left(\gamma \phi(M)\right) = \phi(M),\]
d'où le résultat : $u^{-1}(L)$ est une HLPLI de $X$ et on a $h(u^{-1}(L)) = M$.
\end{proof}

\begin{prop}
Pour presque tout $\xi \in \mathscr{D}$, $\mathscr{H}(\xi)$ est infini.
En particulier, pour presque tout $\xi \in \mathscr{D}$ et toute surface K3 marquée $(X, \phi)$ de période $\xi$, $X$ contient une infinité d'hypersurfaces Levi-plates linéaires dont les groupes d'homologie sont distincts.
\label{infinite-hlpl}
\end{prop}
\begin{proof}
Soit $E \subset \mathscr{E}$ un ensemble fini. Soit $X = \left\{\xi \in \mathscr{D}\,:\,\mathscr{H}(\xi) = E\right\}$. C'est un ensemble mesurable de $\mathscr{D}$ tel que si $\gamma \in \Gamma$, on a ou bien $X = \gamma X$, ou bien $X$ et $\gamma X$ sont disjoints, suivant que l'on a $\gamma E = E$ ou non. Ceci entraîne que $X$ est de mesure nulle ou pleine, par la proposition \ref{prop-pavage} ci-dessous et le fait que l'action de $\Gamma$ sur $\mathscr{D}$ s'identifie à l'action par translation d'un réseau sur $SO^\circ(3, 19)/\left(SO(2)\times SO(1, 19)\right)$.

Or si $X$ est de mesure pleine, on a $\gamma E = E$ pour tout $\gamma \in \Gamma$. Comme $E$ est fini, il existe un sous-groupe d'indice fini de $\Gamma$ qui fixe tous les éléments de $E$. Ceci entraîne que $E$ est vide. Or dans ce cas $X$ n'est pas de mesure pleine puisqu'il ne contient pas le borélien $B$. Dans tous les cas, $X$ est de mesure nulle. Par réunion dénombrable, l'ensemble des $\xi$ tels que $H(\xi)$ est fini est de mesure nulle.
\end{proof}

\begin{prop}
Soit $P$ une partie borélienne de $\mathscr{D}$ telle que pour tous $\gamma_1, \gamma_2 \in \Gamma$, ou bien $\gamma_1P = \gamma_2P$, ou bien $\gamma_1P \cap \gamma_2P = \varnothing$, à des ensembles de mesure nulle près.
Alors $P$ est de mesure nulle ou pleine.
\label{prop-pavage}
\end{prop}
Cette proposition est une conséquence de la suivante, plus générale.
\begin{prop}
Soit $G$ un groupe de Lie connexe presque simple à centre fini. Soit $\mu$ une mesure de Haar de $G$ et $\Gamma$ un réseau de $G$.

Soit $H$ un sous-groupe d'adhérence non compacte de $G$ contenant un élément hyperbolique.

Soit $P$ une partie borélienne de $G$, invariante par $H$ à droite et telle que pour tous $\gamma_1, \gamma_2 \in \Gamma$, ou bien $\gamma_1P = \gamma_2P$, ou bien $\gamma_1P \cap \gamma_2P = \varnothing$, à des ensembles de mesure nulle près.
Alors ou bien $P$ est de mesure nulle, ou bien $P$ est de mesure pleine. 
\end{prop}
Ici, on dit qu'un élément $a$ est \emph{hyperbolique} s'il existe $X \in \mathfrak{g}$ avec $\mathrm{ad}\,X$ diagonalisable sur $\R$ tel que $a = \exp(X)$.
\begin{proof}
Soit $B$ le complémentaire de la réunion des $\gamma P$, pour $\gamma \in\Gamma$. C'est un ensemble invariant par $\Gamma$ à gauche et par $H$ à droite, à des ensembles de mesure nulle près, donc il est de mesure nulle ou pleine. Si $B$ est de mesure pleine, alors $P$ est de mesure nulle. Dans la suite, on suppose donc que $B$ est de mesure nulle.

On note $\mu$ une mesure de Haar à gauche sur $G$.  La mesure de Haar invariante à gauche est aussi invariante à droite, de sorte qu'on parle simplement de mesure de Haar sans plus de précision. On note $\nu$ la mesure invariante par $G$ à droite sur $\Gamma\backslash G$ associée à $\mu$.
Pour tout $g \in G$, soit $\tilde{A}(g) := \bigcup_{\gamma \in \Gamma} \gamma(P\cap Pg)$ et $A(g)$ l'image de $\tilde{A}(g)$ dans $\Gamma\backslash G$. Enfin, soit 
\begin{align*}S &:= \{g \in G\,:\,\nu(A(g))= \nu(\Gamma\backslash G)\}\\
&= \{g \in G\,:\,\mu(G\setminus\tilde{A}(g)) = 0\}\\
&= \{g\in G\,:\,P = Pg \text{ à des ensembles de mesure nulle près }\}.\end{align*}

De la dernière égalité on déduit le fait suivant :

\begin{fait}
\label{fait-inv}
L'ensemble $S$ est un sous-groupe de $G$. De plus, si $g \in G$ et $h \in S$, on a $A(hg) = A(g)$ et $A(g)h = A(gh)$ à des ensembles de mesure nulle près.
\end{fait}  

\begin{lem}
L'application $g \mapsto \nu(A(g))$ est continue. En particulier, le sous-groupe $S$ de $G$ est fermé.
\end{lem}
\begin{proof}
Soit $\mathscr{F} \subset G$ un domaine fondamental borélien de $\Gamma$. Soit $\mathscr{X}$ un ensemble de translatés de $P$ à gauche par $\Gamma$ tel que les éléments de $\mathscr{X}$ sont deux à deux disjoints, à ensembles de mesure nulle près, et leur réunion est de mesure pleine. On a donc $\tilde{A}(g) = \bigcup_{Q \in \mathscr{X}} (Q \cap Qg)$ à un ensemble de mesure nulle près.
Alors \[g \mapsto \nu(A(g)) = \mu\left(\mathscr{F} \cap \tilde{A}(g)\right) = \sum_{Q \in \mathscr{X}}\mu\left(\mathscr{F} \cap Q \cap Qg\right).\]

Pour tout $Q \in \mathscr{X}$, l'application $g \mapsto \mu\left(\mathscr{F} \cap Q \cap Qg\right) = \mu\left(Q \cap \left(\mathscr{F} \cap Q\right)g^{-1}\right)$ est continue. En effet, pour tous $g, h \in G$, on a
\[\left|\mu\left(Q \cap \left(\mathscr{F} \cap Q\right)g^{-1}\right) - \mu\left(Q \cap \left(\mathscr{F} \cap Q\right)h^{-1}\right)\right| \leq \norm{1_{(\mathscr{F} \cap Q)g^{-1}} - 1_{(\mathscr{F} \cap Q)h^{-1}}}_{L^1(G, \mu)},\]qui tend vers $0$ quand $g$ tend vers $h$ par un raisonnement classique par densité et invariance de $\mu$ (ici, $1_X$ désigne bien entendu la fonction indicatrice de $X$, pour tout borélien $X$).

Finalement, par convergence dominée, l'application $g \mapsto \sum_{Q \in \mathscr{X}} \mu\left(\mathscr{F} \cap Q \cap Qg^{-1}\right)$ est bien continue.
\end{proof}

Dans la suite, on note $\mathfrak{s}$ et $\mathfrak{g}$ les algèbres de Lie respectives de $S$ et $G$, avec $\mathfrak{s} \subset \mathfrak{g}$.

\begin{lem}[Phénomène de Mautner]
\label{mautner}
Soient $a \in S$ et $g\in G$ tels que $a^nga^{-n} \longrightarrow 1$ quand $n \longrightarrow \infty$. Alors $g$ appartient à $S$.
\end{lem}
\begin{proof}
En effet, on a, en utilisant le fait \ref{fait-inv} :
\[\nu(A(g)) = \nu(A(g)a^{-n}) = \nu(A(a^nga^{-n})) \longrightarrow_{n \longrightarrow \infty} \nu(\Gamma\backslash G).\]
\end{proof}

On peut maintenant conclure par un raisonnement classique (voir par exemple \cite[théorème 1.6.1]{bekka2008kazhdan}).
Soit $a = \exp(X)$ un élément hyperbolique non trivial. Pour $\lambda\in \R$, soit $\mathfrak{g}^\lambda \subset \mathfrak{g}$ le sous-espace propre de $\mathrm{ad}\,X$ correspondant à la valeur propre $\lambda$. Soient $\mathfrak{g}^{+} = \bigoplus_{\lambda > 0} \mathfrak{g}^\lambda$ et $\mathfrak{g}^{-} = \bigoplus_{\lambda < 0} \mathfrak{g}^\lambda$. 

\begin{fait}
La sous-algèbre de Lie engendrée par $\mathfrak{g}^{+} \cup \mathfrak{g}^{-}$ est égale à $\mathfrak{g}$.
\end{fait}
\begin{proof}En effet, si l'on note $\mathfrak{l}$ la sous-algèbre de Lie engendrée, la relation $[\mathfrak{g}^\lambda, \mathfrak{g}^\mu] \subset \mathfrak{g}^{\lambda + \mu}$, valable pour tous $\lambda, \mu \in \R$, montre que $\mathfrak{l}$ est un idéal. On a donc $\mathfrak{l} = \mathfrak{g}$ ou $\mathfrak{l} = 0$. Dans ce dernier cas, on a $\mathrm{ad}\,X = 0$, donc $X = 0$ par simplicité, et ainsi $a = 1$, ce qui est faux. Donc $\mathfrak{l} = \mathfrak{g}$.
\end{proof}

\begin{fait}
On a $\mathfrak{g}^+ \cup \mathfrak{g}^{-} \subset \mathfrak{s}$.
\end{fait}
\begin{proof}
Soit $Y \in \mathfrak{g}^{\lambda}$ et notons $g = \exp(Y)$. Alors \[a^nga^{-n} = \exp((\mathrm{Ad}\,a^n)Y) = \exp(e^{n(\mathrm{ad}\,X)}Y) = \exp(e^{n \lambda}Y).\]
Ainsi, lorsque $n \longrightarrow \pm\infty$ (suivant le signe de $\lambda$), $a^nga^{-n} \longrightarrow e$. Par le phénomène de Mautner (lemme \ref{mautner}), on en déduit le résultat, puisque $a \in S$.
\end{proof}
Par les deux faits précédents, on a donc $\mathfrak{s} = \mathfrak{g}$, d'où $S = G$ puisque $G$ est connexe. Cela entraîne que pour tout $g \in G$, $P = Pg$ à un ensemble de mesure nulle près. Donc $P$ est de mesure pleine.

\end{proof}
\begin{rem}
Cette preuve montre qu'on ne peut pas avoir de pavage mesurable non trivial de $G/H$ dont les tuiles sont permutées par un réseau de $G$. On peut montrer que le résultat est vrai dès que $H$ n'est pas d'adhérence compacte \cite{note-pavage}.\end{rem}
\section{Densité des classes d'Ahlfors presque partout} 
\label{section:ahlfors}
\paragraph{Classes d'Ahlfors. }
Soit $L$ une hypersurface Levi-plate linéaire de $X$. Considérons le feuilletage de Cauchy-Riemann sur $L$. Il est défini par une forme différentielle réelle $\eta$ de degré $1$, unique à une constante multiplicative près, qui donne une classe dans $H^1(L, \R)$.

Or l'espace $H^1(L, \R)$ s'identifie par dualité de Poincaré à $H_2(L, \R)$, que l'on peut voir dans $H_2(X, \R)$ par inclusion.  Par dualité de Poincaré, ce groupe s'identifie à $H^2(X, \R)$. On définit donc à partir de $\eta$ une classe $\alpha$ dans $H^2(X, \R)$. 

La forme différentielle $\eta$ est unique à une constante multiplicative près, de sorte que la classe de $\alpha$ dans $\mathbb{P}(H^2(X, \R))$ est bien définie. On l'appellera \emph{classe d'Ahlfors} de $L$, et on la notera $\alpha(L)$. On appellera aussi classe d'Ahlfors un élément non-nul de cette droite.
 Notons enfin que l'on a la propriété de naturalité suivante : si $f \colon X' \to X$ est un biholomorphisme, alors $f^{-1}(L)$ est naturellement une hypersurface Levi-plate et l'on a $\alpha(f^{-1}(L)) = f^{*}\alpha(L)$.

\begin{exemple}
\label{ex-ahlfors}
On suppose que l'on a une courbe elliptique linéarisable $C$ dans une surface complexe $X$, avec la représentation $\rho$ associée au fibré normal de $C$ dans $X$. Soit $V$ un voisinage de $C$ dans $X$ biholomorphe à $C \times_\rho \mathbb{A}(r, R)$ où $r, R$ sont des réels tels que $0 < r < 1 < R$ et $L$ correspondant à $C \times_\rho \mathbb{S}^1$. On reprend les notations de l'exemple \ref{hlpl-elliptique}.

L'hypersurface $\{\mathrm{Im}\,v = 0\}$ de $\C^2/L'$ est analytiquement difféomorphe à $\R^3/\Z^3$, via l'application induite par
\begin{align*}
\R^3 &\to \C^2\\
(x, y, z) &\mapsto (\omega_1 x + \omega_2 y, z + px + qy).
\end{align*}

Le feuilletage de Cauchy-Riemann de $C\times_\rho \mathbb{S}^1$ est défini dans les coordonnées locales $(\zeta, v)$ sur $C\times_\rho \mathbb{S}^1$ par $dv = 0$. Dans les coordonnées $(x, y, z)$ sur $\R^3/\Z^3$, il est donc défini par $de^{2i\pi(z + px + qy)} = 0$, c'est-à-dire par $dz + pdx + qdy = 0$. Notons que cette forme dépend du choix de $(p, q)$, puisqu'on aurait pu choisir un autre difféomorphisme de $C \times_\rho \mathbb{S}^1$ avec $\R^3/\Z^3$.

Notons désormais $\alpha$, $\beta$ et $\gamma$ les lacets images dans $\R^3/\Z^3$ des chemins allant de $0$ à $1$ dans chacun des facteurs de $\R^3$ respectivement.
Cherchons l'image de $dz + pdx +qdy$, vu comme élément de $H^1(L, \R)$, dans $H_2(L, \R)$. On peut vérifier que l'on a pour toute forme $\omega$ de degré $2$ sur $L$
\begin{align*}
\int_{\alpha\times\beta} \omega &= \int dz \wedge \omega, \\
\int_{\beta\times\gamma} \omega &= \int dx \wedge \omega \end{align*}
et
\[\int_{\gamma\times\alpha} \omega = -\int dy\wedge \omega.
\]
donc ces trois $2$-cycles représentent les classes de cohmologie $[dz]$, $[dx]$ et $-[dy]$ via l'isomorphisme de De Rham.
On peut considérer les trois $2$-cycles $\alpha\times\beta$, $\beta\times\gamma$ et $\gamma\times\alpha$ dans $X$. Soient donc $A_{\alpha\beta}$, $A_{\beta\gamma}$ et $A_{\gamma\alpha}$ les duaux de Poincaré dans $X$ respectifs.
La classe d'Ahlfors de l'hypersurface Levi-plate linéaire $L$ est alors donnée par $A_{\alpha\beta} + pA_{\beta\gamma} - qA_{\gamma\alpha}$.

\begin{prop}
La classe d'Ahlfors de $L$ dans $X$ est d'auto-intersection nulle, et elle définit un élément de $H^{1, 1}(X, \R)$.\label{prop-classe-ahlfors}
\end{prop}
\begin{proof}Pour la première propriété, il suffit de remarquer que tous les produits entre $A_{\alpha\beta}$, $A_{\beta\gamma}$ et $A_{\gamma\alpha}$ sont nuls. On peut aussi le voir en notant que l'on peut déformer $L$ en un tore disjoint dans $V$.

Pour la seconde propriété, on peut de même calculer explicitement la classe d'Ahlfors en se plaçant dans un ouvert $\C^2/L'$ qui contient une copie de $L$ comme dans l'exemple \ref{hlpl-elliptique}. Cela permet de vérifier que la classe d'Ahlfors est dans $H^{1, 1}(X, \R)$. Comme le calcul est semblable à celui qui a déjà été fait, on ne le refait pas ici.
\end{proof}

\begin{rem}
La classe d'Ahlfors peut être calculée à partir de la classe de cohomologie d'un courant d'Ahlfors (voir Brunella \cite{brunella1999courbes} par exemple) défini par une courbe entière donnée par une feuille du feuilletage de Cauchy-Riemann. 
\end{rem}
\end{exemple}

\paragraph{Lien avec le marquage. }

On peut montrer un analogue de la proposition \ref{prop:equiv-homologie}, c'est-à-dire que l'on construit une application qui associe à une période l'ensemble des $\alpha(L)$, vues par le marquage, où $L$ est une HLPL d'une surface obtenue par action d'un élément du groupe $\Gamma$ sur une surface de Koike-Uehara. La preuve est essentiellement la même.
\begin{prop}Il existe une application $\mathscr{A} \colon \mathscr{D} \to \mathscr{P}(\mathbb{P}(\Lambda \otimes \R))$ qui est $\Gamma$-équivariante et telle que pour tout $\xi \in \mathscr{D}$, on a $\mathscr{A}(\xi) \subset \left\{\R v \in \mathbb{P}(\Lambda\otimes \R)\,:\,v \perp P(\xi),\,q(v) = 0\right\}$. Enfin, si $\xi \in A$, qui est de mesure pleine, $\mathscr{A}(\xi)$ est non-vide.

De plus, pour toute $\xi \in \mathscr{G}$, toute surface K3 marquée $(X, \phi)$ de période $\xi$, et tout $\alpha \in \phi_{\R}^{-1}(\mathscr{A}(\xi)) \subset \mathbb{P}(H^2(X, \R))$, il existe une hypersurface Levi-plate irrationnelle $L \subset X$ de classe d'Ahlfors $\alpha$.
\label{prop-ahlfors}
\end{prop}

\begin{proof}

Pour tout $\xi \in B$, soit $(X, \phi)$ la surface de Koike-Uehara associée et $L$ une de ses hypersurfaces Levi-plates linéaires associées. On peut considérer sa classe d'Ahflors $\alpha(L) \in \mathbb{P}(H^2(X, \R))$. 

Via le marquage qui donne par tensorisation une application $\phi_{\R} \colon H^2(X, \R) \to \Lambda\otimes\R$, on en déduit un élement $a(\xi)$ de $\mathbb{P}(\Lambda \otimes \R)$ qui est orthogonal à $P(\xi)$, puisque l'on a en fait $\alpha(L) \in \mathbb{P}(H^{1, 1}(X, \R))$. 

Étant donné $\xi \in \mathscr{D}$, considérons comme plus haut l'ensemble $\Gamma_{\xi} = \{\gamma\in\Gamma\;:\;\gamma\xi\in B\}$ et notons $\mathscr{A}(\xi) = \{\gamma^{-1}a(\gamma\xi)\;:\;\gamma\in\Gamma_{\xi}\}$. La propriété d'équivariance est alors claire. La dernière partie de l'énoncé se prouve comme la proposition \ref{prop:equiv-homologie}, grâce au théorème de Torelli (théorème \ref{torelli}).

Le fait que $\mathscr{A}(\xi)$ est non vide si $\xi \in A$ est une conséquence de la définition. Quant au fait que $\mathscr{A}(\xi) \subset \left\{\R v \in \mathbb{P}(\Lambda\otimes \R)\,:\,v \perp P(\xi),\,q(v) = 0\right\}$, c'est une conséquence de la proposition \ref{prop-classe-ahlfors} (le marquage envoie le $H^{1, 1}$ dans l'orthogonal de la période, donc du plan réel associé).

\end{proof}
\begin{rem}
Reprenons les notations de l'exemple \ref{ex-ahlfors}. 
Posons pour tout $(p, q) \in \R^2$, $v_{(p, q)} := A_{\alpha\beta} + pA_{\beta\gamma} - qA_{\gamma\alpha}$. Alors si $\xi \in v_{(p, q)}^{\perp}$, on a $\alpha(L) = v_{(p, q)}$. Cela suit du calcul de l'exemple \ref{ex-ahlfors} et de la définition du marquage de Koike-Uehara \cite[section 3]{koike2019gluing}, qui est défini de telle sorte que les différentes notations $A_{\alpha\beta}$, $A_{\beta\gamma}$ et $A_{\gamma\alpha}$ soient cohérentes.
\end{rem}

\paragraph{Flot géodésique. } Dans ce paragraphe, on montre la proposition \ref{coro-ahlfors}, en adaptant un argument de Ghys \cite{ghys1995topologie}, repris également par Martinez-Matsumoto-Verjovsky \cite{martinez}, qui est très proche de l'argument donné ici. On utilise cependant le flot géodésique et non le mouvement brownien, plus simple dans le cadre plus restreint qui est le nôtre.
\begin{defn}Étant donné $\xi \in \mathscr{D}$, soient \[\tilde{\mathscr{F}}_\xi := \left\{\R v \in \mathbb{P}(\Lambda\otimes \R)\,:\,v \perp P(\xi),\,q(v) > 0\right\}\]et
\[\tilde{M} = \left\{(\xi, p) \in \mathscr{D}\times \mathbb{P}(\Lambda\otimes\R)\,:\,p \in \tilde{\mathscr{F}}_\xi\right\}.\]
\end{defn}

\begin{fait}Pour tout $\xi \in \mathscr{D}$, $\tilde{\mathscr{F}}_\xi$ est un espace hyperbolique de dimension $19$, dans le modèle projectif, associé à la forme quadratique $q$.

Ceci définit un feuilletage $\tilde{\mathscr{F}}$ de $\tilde{M}$ dont les feuilles sont les $\tilde{\mathscr{F}}_\xi$, $\xi \in \mathscr{D}$.\end{fait} Remarquons qu'avec cette convention, la métrique riemannienne associée sur l'espace tangent à $\R v$, $v \in \Lambda\otimes\R$, identifié à $v^\perp$, est induite par $-q$.

Par la proposition \ref{prop-ahlfors}, on a :
\begin{fait}
Pour tout $\xi \in \mathscr{D}$, on a $\mathscr{A}(\xi) \subset \partial \tilde{\mathscr{F}}_\xi$, où l'on a noté $\partial \tilde{\mathscr{F}}_\xi$ le bord de l'espace hyperbolique $\tilde{\mathscr{F}}_\xi$, identifié à $\left\{\R v \in \mathbb{P}(\Lambda\otimes \R)\,:\,v \perp P(\xi),\,q(v) = 0\right\}$.
\end{fait}

\begin{lem}
\label{lem-mes-conv}
Pour tout $r > 0$, l'ensemble
\[C^r := \{(\xi, p)\in \tilde{M}\,:\,d(p, \mathrm{Conv}(\overline{\mathscr{A}(\xi)}) \leq r\}\]
est un borélien de mesure strictement positive.

Ici, pour $\xi \in \mathscr{D}$, on a noté $d$ la distance dans l'espace hyperbolique $\tilde{\mathscr{F}}_\xi$ et, si $Q \subset \partial \tilde{\mathscr{F}}_\xi$, alors $\mathrm{Conv}(Q)$ désigne l'enveloppe convexe de $Q$ dans $\tilde{\mathscr{F}}_\xi \cup \partial \tilde{\mathscr{F}}_\xi$ (qui est un espace hyperbolique compactifié).
\end{lem}
Soit $\tilde{N} := \{(\xi, p, v)) \in \tilde{M} \times \left(\Lambda\otimes\R\right)\,:\, v \in T_p\tilde{\mathscr{F}}_\xi, -q(v) = 1\}$ : c'est le fibré tangent unitaire au feuilletage $\tilde{\mathscr{F}}$ de $\tilde{M}$. Notons $\pi \colon \tilde{N} \to \tilde{M}$ la projection associée.

Pour tout $(\xi, x, v) \in \tilde{N}$ et $t \in \R$, soit $\phi_t(\xi, x, v) = (\xi, x', v')$ où $(x', v')$ est obtenu en suivant pendant le temps $t$ la géodésique dans $\tilde{\mathscr{F}}_\xi$ partant de $x$ et de vecteur tangent $v$ au temps zéro, dans $\tilde{\mathscr{F}}_\xi$. On appellera $(\phi_t)_{t\in\R}$ le \emph{flot géodésique} sur $\tilde{N}$. Il passe au quotient en un flot encore appelé \emph{flot géodésique} et noté aussi $(\phi_t)_{t\in \R}$ sur $\Gamma\backslash \tilde{N}$.

Notons que $SO^\circ(q)$ agit transitivement sur $\tilde{N}$. De plus, on peut identifier $SO^\circ(q)$ à $SO^\circ(3, 19)$ de sorte que le stabilisateur d'un point soit $SO(2) \times SO(18)$ via le plongement
\[(A, B) \in SO(2) \times SO(18) \mapsto \begin{pmatrix}A & 0 & 0 \\ 0 & I_2 & 0 \\ 0 & 0 & B\end{pmatrix}.\] On identifie donc $\tilde{N}$ à $SO^\circ(3, 19)/(SO(2) \times SO(18))$. 
Alors $\phi_t$ correspond à la multiplication à droite par
\[\begin{pmatrix}I_2 & 0 & 0 \\ 0 & \theta(t) & 0 \\ 0 & 0 & I_{18}\end{pmatrix},\]
où $\theta(t) = \begin{pmatrix}\cosh(t) & \sinh(t) \\ \sinh(t) & \cosh(t)\end{pmatrix}$. 

On peut munir $\tilde{N}$ de la mesure induite par la mesure de Haar, $SO(2) \times SO(18)$ étant compact ; cela définit une mesure absolument continue qui est invariante à droite par le flot géodésique, puisqu'il est défini par l'action à droite d'un sous-groupe à un paramètre commutant avec $SO(2) \times SO(18)$. Cette mesure est par ailleurs invariante à gauche ; elle définit donc une mesure sur $\Gamma\backslash \tilde{N}$. De plus le théorème de Howe-Moore \ref{howe-moore} entraîne que l'on a :
\begin{lem}Le flot géodésique sur $\Gamma\backslash \tilde{N}$ est ergodique.\end{lem}

On peut désormais montrer la proposition suivante :
\begin{prop}
Pour presque tout $\xi \in \mathscr{D}$, l'ensemble $\mathscr{A}(\xi)$ est dense dans $\partial \tilde{\mathscr{F}}_\xi$.
\end{prop}
\begin{proof}
Soit $r > 0$ quelconque. Soit $C = \Gamma\backslash\pi^{-1}(C^r) \subset \Gamma\backslash \tilde{N}$. Par le théorème de Fubini, c'est un sous-ensemble borélien de mesure positive. En particulier, par l'ergodicité du flot géodésique, pour presque tout $\Gamma(\xi, x, v) \in \Gamma\backslash \tilde{N}$, on a
\[\limsup_{n \rightarrow \infty} \frac{1}{n}\sum_{k = 0}^{n - 1} 1_C\left(\phi_n\left(\Gamma(\xi, x, v)\right)\right) > 0.\]

Par conséquent, de nouveau par le théorème de Fubini, pour presque tout $\xi$ et presque tout $(x, v) \in T\tilde{\mathscr{F}}_\xi$ avec $-q(v) = 1$, il existe des $t$ arbitrairement grands tel que si $\gamma_{x, v}$ désigne la géodésique partant de $x$ de vecteur tangent $v$, alors $\gamma_{x, v}(t)$ est à une distance au plus $r$ de $\mathrm{Conv}(\overline{\mathscr{A}(\xi)})$, puisque $\mathrm{Conv}(\overline{\mathscr{A}(\xi)})$ est équivariant sous l'action de $\Gamma$. Ceci n'est possible que si le point du bord $\lim_{t \rightarrow \infty}\gamma_{x, v}(t) \in \partial \tilde{\mathscr{F}}_\xi$ est dans $C^r(\xi) \cap \partial \tilde{\mathscr{F}}_\xi = \overline{\mathscr{A}(\xi)}$. On en déduit le résultat.
\end{proof}

En combinant ce dernier résultat avec la proposition \ref{prop-ahlfors}, on en déduit finalement le théorème \ref{coro-ahlfors}.

\appendix
\section*{Annexe : Mesurabilité et mesure positive}
\begin{proof}[Preuve du corollaire \ref{thm-construction}]

On part de l'énoncé suivant, démontré par Koike et Uehara.
\begin{thm}[Théorème 1.6 de \cite{koike2019gluing}]Soit $(p, q)$ un couple de réels satisfaisant une condition diophantienne ; rappelons qu'on a noté $v_{(p, q)} = A_{\alpha\beta} + pA_{\beta\gamma} - qA_{\gamma\alpha}$. Alors il existe un ouvert $\Xi_{(p, q)}$ de $\mathscr{D} \cap v_{(p,q)}^\perp$ et une submersion holomorphe propre $\pi_{(p, q)} \colon \mathcal{X}_{(p, q)}\to \Xi_{(p, q)}$ telle que toute fibre est une surface K3 de Koike-Uehara dont l'application des périodes induit l'identité $\Xi_{(p,q)} \to \Xi_{(p, q)}$.
\label{koike-uehara}
\end{thm}
Étant donné $a, b \in \C$, notons $[a, b]$ la matrice de la famille $(a, b)$ dans la base $(1, i)$ de l'espace vectoriel réel $\C$. 

Soit \[U := \left\{\C\xi \in \mathscr{D}\,:\,b_{\gamma}(\xi)\neq 0\text{ et } \left[\frac{b_{\alpha}(\xi)}{b_{\gamma}(\xi)}, -\frac{b_{\beta}(\xi)}{b_{\gamma}(\xi)} \right] \text{ est inversible}\right\}.\]
Par un calcul que l'on ne détaille pas ici, on peut vérifier que c'est un ouvert de $\mathscr{D}$ tel que $\mathscr{D}\setminus U$ est de mesure nulle (c'est en fait une courbe réelle).
Soit maintenant $V \colon U \to \R^2$ l'application
\[\C\xi \mapsto \left[\frac{b_{\alpha}(\xi)}{b_{\gamma}(\xi)}, -\frac{b_{\beta}(\xi)}{b_{\gamma}(\xi)} \right]^{-1}\begin{pmatrix}-1 \\ 0\end{pmatrix}.\]
Par définition, si $\xi \in U \cap v_{(p, q)}^{\perp}$, on a $V(\C\xi) = (p, q)$. De même, on peut vérifier que c'est une submersion.

Posons $U' = U \cap \bigcup_{(p, q) \in \mathrm{Dioph}} \Xi_{(p, q)}$. Étant donné $\xi \in U'$, soit $(p, q) = V(\xi)$. On peut considérer les surfaces $S^{\pm}(\xi)$, les courbes $C^{\pm}(\xi)$ et $\eta^{\pm}(\xi)$ les $2$-formes méromorphes sur $S^{\pm}(\xi)$, comme dans la construction de Koike-Uehara, associées à la période $\xi$, dans la famille $\pi_{(p, q)}$ du théorème. Notons de plus $W^{-}_{\text{max}}(\xi)$ le voisinage tubulaire holomorphe de $C^{\pm}(\xi)$ maximal dans $S^{\pm}(\xi)$ (voir \cite[lemme 2.7]{koike2019gluing}). Alors on peut prendre dans l'énoncé du théorème \ref{koike-uehara}, pour tout $(p, q) \in \R^2$ satisfaisant une condition diophantienne \cite[théorèmes 1.6, 6.2 et 6.4]{koike2019gluing}

\[\Xi_{(p, q)} =\left\{\xi \in v_{(p, q)}^\perp\,:\,b_\beta(\xi) \neq 0,\,\mathrm{Im}\left(\frac{b_{\alpha}(\xi)}{b_{\beta}(\xi)}\right) >0, q(\xi,\overline{\xi}) > \Lambda_{(p, q)}(\xi)\right\},\]
où
\[\Lambda_{(p, q)}(\xi)  := \int_{S^{+}(\xi)\setminus W^{+}_{\text{max}}(\xi)} \eta^{+}(\xi)\wedge\overline{\eta^{-}(\xi)} + \int_{S^{-}(\xi)\setminus W^{-}_{\text{max}}(\xi)} \eta^{-}(\xi)\wedge\overline{\eta^{-}(\xi)}.\]
La preuve du théorème \ref{koike-uehara} repose elle-même sur celle du théorème 8.4 de Koike-Uehara, fondée sur la construction de coordonnées qui donnent le voisinage recherché par l'usage de séries majorantes, dont il s'agit de montrer que le rayon de convergence est strictement positif. Ces opérations peuvent toutes facilement être quantifiées, la seule difficulté étant l'usage d'un théorème de Siegel \cite{siegel1942iteration}. Mais celui-ci donne lui aussi un rayon explicite en fonction de la condition diophantienne et de bornes sur les différentes fonctions en jeu. Cela entraîne que si l'on se restreint à $V^{-1}(\text{Dioph}(C, \alpha))$, la fonction $\xi \mapsto \Lambda_{V(\xi)}(\xi)$ est semi-continue supérieurement.
L'ensemble $U' \cap V^{-1}(\text{Dioph}(C, \alpha))$ est alors un ouvert de $V^{-1}(\mathrm{Dioph}(C, \alpha))$, donc en faisant une réunion dénombrables, on en déduit que $U'$ est un borélien.
\end{proof}

\begin{proof}[Démonstration du lemme \ref{lem-mes-conv}]
Soit $r > 0$. Soit $D$ un ensemble dénombrable de formes linéaires dense dans $(\Lambda\otimes\R)^{*}$. Pour tout $\ell \in D$ et tout $\xi \in \mathscr{D}$, définissons un ensemble $C_\ell(\xi)$ de la façon suivante :
\begin{itemize}
\item si $\ell$ délimite deux demi-espaces ouverts dans $\Hp_\xi$, dont l'un, noté $E$, contient $\mathscr{A}(\xi)$, alors
on pose $C_\ell(\xi) := E$ ;
\item sinon, on pose $C_\ell(\xi) := \Hp_\xi$.
\end{itemize}
Posons enfin $C_\ell^r(\xi) := \{x\in\Hp_\xi\,;\,d(x, C_\ell(\xi)) \leq r\}$ et $C^r(\xi) := \bigcap_{\ell \in D} C_\ell^r(\xi)$. L'ensemble $C^r(\xi)$ est l'ensemble des points à distance inférieure ou égale à $r$ de l'enveloppe convexe de l'adhérence de $\mathscr{A}(\xi)$.

L'ensemble
\[C_\ell^r := \{(\xi, x) \in M\,:\,x \in C_\ell^r(\xi)\}\]
est un borélien. En effet, supposons que $\ell$ délimite deux demi-espaces ouverts dans $\tilde{\mathscr{F}}_\xi$, et soit $E$ l'un des deux. La propriété que $E$ contient $\mathscr{A}(\xi)$ s'écrit :
\[\text{pour tout }\gamma \in \Gamma\text{ tel que }\gamma \xi \in A, a(\gamma\xi) \in \gamma E.\]

L'ensemble $\bigcap_{\ell \in D} C^r_\ell$ est aussi un borélien. Or
\[\bigcap_{\ell \in D} C^r_\ell = \{(\xi, x)\,:\,x \in C^r(\xi)\} = \{(\xi, x)\,:\,d(x, \mathrm{Conv}(\overline{\mathscr{A}(\xi))}) \leq r\} = C^r,\] d'où la première partie de l'énoncé.

Pour la seconde partie, remarquons que pour tout $\xi \in \mathscr{D}$, $C^r(\xi)$ est de mesure strictement positive dès que $\mathscr{A}(\xi)$ est de cardinal au moins $2$. Comme ceci est vrai pour presque tout $\xi$ par la proposition \ref{infinite-hlpl}, et que $M$ est fibré par les $\tilde{\mathscr{F}}_\xi$, on en déduit par le théorème de Fubini que $C^r$ est de mesure strictement positive.
\end{proof}

\bibliographystyle{plain}
\bibliography{HLPLfinal}
\end{document}